\documentclass[12pt]{amsart}
\usepackage{amsmath,amssymb,amsthm}
\usepackage{setspace}
\usepackage[colorlinks=true,pagebackref,hyperindex,citecolor=green,linkcolor=red]{hyperref}
\usepackage{moreverb}
\usepackage{mathrsfs}
\usepackage{url}
\usepackage[all]{xy}
\usepackage{lscape}
\usepackage{tikz}
\usetikzlibrary{arrows}
\usepackage{xcolor}
\usepackage{array}

\newcommand{\ignore}[1]{}

\newtheorem{theorem}{Theorem}[section]
\newtheorem{lemma}[theorem]{Lemma}
\newtheorem{corollary}[theorem]{Corollary}
\newtheorem{proposition}[theorem]{Proposition}

\theoremstyle{definition}
\newtheorem{definition}[theorem]{Definition}
\newtheorem{example}[theorem]{Example}

\theoremstyle{remark}
\newtheorem{remark}[theorem]{Remark}

\numberwithin{equation}{section}

\input xy
\xyoption{all}

\newcommand{\bC}{\mathbb{C}}

\newcommand{\bN}{{\mathbb{N}}}
\newcommand{\bZ}{{\mathbb{Z}}}

\newcommand{\bQ}{{\mathbb{Q}}}
\newcommand{\bR}{\mathbb{R}}

\definecolor{grey}{rgb}{0.75,0.75,0.75}
\definecolor{orange}{rgb}{1.0,0.5,0.5}
\definecolor{brown}{rgb}{0.5,0.25,0.0}
\definecolor{pink}{rgb}{1.0,0.5,0.5}

\newcommand{\adj}{\mathrm{Adj}}

\newcommand{\fM}{{\mathfrak m}}

\newcommand{\fa}{\mathfrak{a}}

\newcommand{\cC}{{\mathcal C}}

\newcommand{\cN}{{\mathcal N}}
\newcommand{\cO}{{\mathcal O}}

\newcommand{\cR}{{\mathcal R}}

\newcommand{\lra}{{\longrightarrow}}

\newcommand{\cn}{{\pmb{c}}}
\newcommand{\alphan}{{\pmb{\alpha}}}

\newcommand{\lambdab}{\pmb{\lambda}}
\newcommand{\lambdapb}{\pmb{\lambda'}}

\newcommand{\fab}{\pmb{\mathfrak{a}}}

\newcommand{\A}{\mathfrak{a}}

\newcommand{\Z}{\mathbb{Z}}
\newcommand{\Q}{\mathbb{Q}}
\newcommand{\R}{\mathbb{R}}

\newcommand{\Oc}{\mathcal{O}}
\newcommand{\m}{\mathfrak{m}}

\newcommand{\J}{\mathcal{J}}

\usepackage{pgfplots}

\pgfplotsset{compat=1.14}

\begin{document}

\title[Multiplicities of jumping points]{Multiplicities of jumping points for mixed multiplier ideals}
\author[M. Alberich]{Maria Alberich-Carrami\~nana}

\author[J. \`Alvarez ]{Josep \`Alvarez Montaner}

\address{Deptartament de Matem\`atiques\\
Universitat Polit\`ecnica de Catalunya\\ Av. Diagonal 647, Barcelona
08028, Spain} \email{Maria.Alberich@upc.edu, Josep.Alvarez@upc.edu}

\author[F. Dachs] {Ferran Dachs-Cadefau }
\address{Institut f\"ur Mathematik , Martin-Luther-Universit\"at Halle-Wittenberg\newline 06099 Halle (Saale), Germany}
\email{Ferran.Dachs-Cadefau@mathematik.uni-halle.de}

\author[V. Gonz\'alez]{V\'ictor Gonz\'alez-Alonso}
\address{Institut f\"ur Algebraische Geometrie\\ Leibniz Universit\"at Hannover \\
Welfengarten 1, 30167 Hannover, Germany} \email{gonzalez@math.uni-hannover.de}

\thanks{All four authors are partially supported by
Spanish Ministerio de Econom\'ia y Competitividad
MTM2015-69135-P. VGA is partially supported by ERC StG 279723 “Arithmetic of algebraic surfaces” (SURFARI).
MAC and JAM are also supported by Generalitat de Catalunya SGR2017-932 project and they
are with the Barcelona Graduate School of Mathematics (BGSMath).
MAC is also with the  Institut de Rob\`otica i Inform\`atica Industrial (CSIC-UPC)}



\begin{abstract}
In this paper we make a systematic study of the multiplicity of the jumping points associated
to the mixed multiplier ideals of a family of ideals in a complex surface with rational singularities.
In particular we study the behaviour of the multiplicity by small perturbations of the jumping points.
We also introduce a Poincar\'e series for mixed multiplier ideals and prove its rationality. Finally,
we study the set of divisors that contribute to the log-canonical wall.
\end{abstract}

\maketitle

\section{Introduction}

Let $X$ be a complex surface with a rational singularity at a point $O\in X$
and $\cO_{X,O}$ its corresponding local ring.  Let $\fa \subseteq \cO_{X,O}$ be an $\fM$-primary
ideal where $\fM=\fM_{X,O}$ is the maximal ideal of $\cO_{X,O}$. Then, for any real exponent $c>0$,
we may consider its corresponding \emph{multiplier ideal} $\J(\fa^c)$. Indeed, the multiplier ideals form a discrete nested sequence
\begin{equation} \label{sequence_mult}
\Oc_{X,O}\varsupsetneq\J(\A^{\lambda_1})\varsupsetneq\J(\A^{\lambda_2})\varsupsetneq\ldots\varsupsetneq\J(\A^{\lambda_i})\varsupsetneq\ldots
\end{equation}
indexed by an increasing sequence of rational numbers $0 < \lambda_1 < \lambda_2 < \ldots$ such that
$\J(\A^{\lambda_i})=\J(\A^c)\varsupsetneq\J(\A^{\lambda_{i+1}})$ for any $c \in \left[\lambda_i,\lambda_{i+1}\right)$. The
$\lambda_i$ are the so-called {\em jumping numbers} of the ideal $\fa$.
Ein, Lazarsfeld, Smith and Varolin \cite{ELSV04}, using the fact that the multiplier ideals are $\fM$-primary as well,
defined the {\em multiplicity} of a point $c \in \bR$ as
\begin{equation} \label{eq-def-mult-general}
m\left(c\right) := \dim_{\bC}\frac{\J\left(\fa^{c-\varepsilon}\right)}{\J\left(\fa^c\right)}
\end{equation}
for $\varepsilon>0$ small enough. With this definition, it is clear that $c$ is a jumping number if and only if $m\left(c\right)>0$.
A way to encode the information provided by the filtration of ideals (\ref{sequence_mult})
is by means of its {\em Poincar\'e series of multiplier ideals}
\begin{equation} \label{eq-def-Poinc}
P_\A (t)= \sum_{c\in \bR_{>0}} m(c) \hskip 1mm t^{c}
\end{equation}
introduced by Galindo and Montserrat in \cite{GM10}. Actually, they proved that this series is a rational function
when $X$ is smooth and $\fa$ is simple. In \cite{ACAMDCGA13} we gave a systematic study of multiplicities and
proved the rationality of the Poincar\'e series for any ideal $\fa$ in a complex surface with a rational singularity at $O$.

\vskip 2mm

Whenever we extend to the case of {\it mixed multiplier ideals}
$\J\left(\fab^{\mathbf{c}}\right):=\J\left(\fa_1^{c_1}\cdots\fa_r^{c_r}\right)$ associated to a tuple of
$\fM$-primary ideals $\fab:=(\fa_1,\ldots,\fa_r )\subseteq (\cO_{X,O})^r$ and a point $\mathbf{c}:=(c_1,\ldots,c_r)$  in the positive
orthant  $\R_{\geqslant0}^r$, things become a little bit more trickier. Instead of having a partition of the positive real line into
intervals defined by the jumping numbers  where the multiplier ideals are constant, we get a partition of the positive real orthant
into {\it constancy regions} whose boundary is described by the so-called {\it jumping walls}.  We point out that only a few results on mixed
 multiplier ideals are available in the literature.
Libgober and Musta\c{t}\v{a} \cite{LM11} studied  properties of the {\it log-canonical wall}, i.e. the jumping wall
associated to $\lambdab_0=(0,\ldots,0)$.
Naie in \cite{Nai13}  describes a nice property that jumping walls must satisfy.
Cassou-Nogu\`es and Libgober study in \cite{CNL11,CNL14}  mixed multiplier ideals and jumping
walls associated to germs of plane curves, under the analogous notion of {\it ideals of quasi-adjunction} and {\it faces of quasi-adjunction} (see \cite{Lib02}). 


\vskip 2mm

We may define the multiplicity
of any given point $\cn \in \R_{\geqslant0}^r$  and we say that it is a {\it jumping point}, if and only if $m\left(\cn\right)>0$.
In particular, jumping points  lie on  jumping walls. The most natural generalization of a Poincar\'e series for mixed multiplier ideals is to consider
a filtration of ideals $\J\left(\fab^{\mathbf{c}}\right)$ indexed by points over a ray $L$ in $\R_{\geqslant0}^r$ with direction vector in $\bQ^r_{\geqslant0}$ and set
\begin{equation} \label{eq-def-Poinc2}
P_{\fab} (\underline{t}; L)= \sum_{\cn \in  L} m(\cn) \hskip 1mm \underline{t}^{\cn}
\end{equation}
where $\underline{t}^{\cn}:=t_1^{c_1}\cdots t_r^{c_r}$.
We have to mention that the multiplicity
is sensitive to small perturbation of a given point. Indeed, if we consider a sequence of mixed multiplier ideals
 $$\J \left(\fab^{{\cn_0}}\right) \supsetneq \cdots \supsetneq \J \left(\fab^{\cn_{i-1}}\right) \supsetneq
 \J \left(\fab^{\cn_{i}}\right)  \supsetneq   \J \left(\fab^{\cn_{i+1}}\right) \supsetneq  \cdots $$
 indexed by jumping points $\{\cn_i\}_{i\geqslant 0}$ over a ray $L$ and perturb minimally this ray, for example taking a
 parallel ray $L'$ that is close enough,
%
 then the sequence of mixed multiplier ideals indexed
 by jumping points in $L'$ may vary (see Example \ref{rays}) and thus, the corresponding Poicar\'e series also varies.

\vskip 2mm

The organization of the paper is as follows. In Section \ref{Sec1} we recall all the basics on mixed multiplier ideals.
In Section \ref{multiplicities} we extend the results of \cite{ACAMDCGA13} to this setting. Namely, we make a systematic study of
the multiplicities of points in the positive orthant. The main result is Theorem \ref{thm:multiplicity1_cMMI} where we
give a precise formula for the multiplicity. We also prove in Theorem \ref{thm:rational} that the Poincar\'e series associated to a ray is a
rational function. In Section \ref{perturbations} we study the variation of the multiplicity of a jumping point by
small perturbations. We prove that this multiplicity does not vary for points in the interior (with the Euclidean
topology) of a $\cC$-facet (see Proposition \ref{interior}) but it does so in a controlled way at the intersection of $\cC$-facets
(see Theorem \ref{intersection}). In Section \ref{Sec4} we study the exceptional divisors that contribute to the log-canonical wall.
Our main result is Theorem \ref{thm:Pi} where we establish (except for a very particular case considered in Proposition \ref{lem:2rdiv}) a one-to-one correspondence between the facets of the log-canonical
wall and the exceptional divisors of the so-called {\it Newton nest} (see Definition \ref{Newton}) generalizing
a result of Cassou-Nogu\`es and Libgober \cite[Theorem 4.22]{CNL14}.

\section{Mixed multiplier ideals} \label{Sec1}

Let $X$ be a complex surface with at most a rational singularity at a point $O \in X$ (see Artin
\cite{Art66} and  Lipman \cite{Lip69} for details) and $\fM =\fM_{X,O} $ be the maximal ideal of the local ring $\cO_{X,O}$ at $O$.
Given a tuple of $\fM$-primary ideals $\fab=(\fa_1,\ldots,\fa_r )\subseteq (\cO_{X,O})^r$ we will
consider a common {\em log-resolution}, that is, a birational
morphism $\pi: X' \rightarrow X$ such that $X'$ is smooth, $\fa_i\cdot\cO_{X'} = \cO_{X'}\left(-F_i\right)$ for some effective
Cartier divisors $F_i$,  $i=1,\dots , r$ and $\sum_{i=1}^r F_i+E$ is a divisor with simple normal crossings, where $E = Exc\left(\pi\right)$ is the exceptional locus. Actually, the divisors $F_i$ are supported on the exceptional locus since the ideals are $\fM$-primary.
The {\it fundamental cycle} is the
unique smallest non-zero effective divisor $Z$
(with exceptional support) such that
$Z\cdot E_i \leqslant 0$ {for every }  $i=1,\dots, r.$
The fundamental cycle satisfies
$\fM\cdot\cO_{X'} = \cO_{X'}\left(-Z\right)$ (see \cite[Theorem 4]{Art66}). We point out that any effective divisor with integer coefficients
$\widetilde{D}$ is called
{\it antinef}  if $\widetilde{D}\cdot E_i \leqslant 0$ for every $i=1,\dots, r.$ Indeed, for any effective divisor
$D$ there exists a unique minimal antinef divisor $\widetilde{D}$ satisfying $D\leqslant \widetilde{D}$
that is called the {\it antinef closure} of $D$.
It can be computed using an inductive procedure called {\it unloading} (see \cite{Reg96} and \cite{ACAMDC16} for details).


\vskip 2mm

Since the point $O$ has a rational singularity, the
exceptional locus $E$ is a tree of smooth rational curves
$E_1,\dots,E_s$. Moreover,  the matrix of
intersections $\left(E_i\cdot E_j\right)_{1\leqslant i,j \leqslant
s}$ is negative-definite. For any exceptional component $E_j$, we define the {\em excess} of
$\fa_i$ at $E_j$ as $\rho_{i,j} = - F_i \cdot E_j$.  We also recall the following notions:
\begin{enumerate}
\item[$\cdot$] A component $E_j$ of $E$ is a {\em rupture} component if
it intersects at least three more components of $E$ (different from
$E_j$).

\item[$\cdot$] We say that $E_j$ is {\em dicritical} if $\rho_{i,j} > 0$ for some $i$.
Such components  correspond to {\it Rees valuations}
(see \cite{Lip69}).
\end{enumerate}

\vskip 2mm

We define the {\it mixed multiplier ideal}  at a point
$\mathbf{c}:=(c_1,..,c_r) \in \R_{\geqslant0}^r$
as\footnote{By an abuse of notation, we will also denote
$\J\left(\fab^{\cn}\right)$ its stalk at $O$ so we will omit
the word ''sheaf'' if no confusion arises.}
\begin{equation} \label{eq:mmi}
\J\left(\fab^{\mathbf{c}}\right):=\J\left(\fa_1^{c_1}\cdots\fa_r^{c_r}\right)
= \pi_*\cO_{X'}\left(\left\lceil K_{\pi} - c_1 F_1- \cdots -
c_r F_r \right\rceil\right)
\end{equation}

\noindent where  $\lceil \cdot \rceil$ denotes the
 {\em round-up} and the {\it relative canonical divisor} $$ K_{\pi}=
\sum_{i=1}^s k_j E_j $$ is the $\bQ$-divisor on $X'$ supported on the
exceptional locus $E$  characterized by the property $\left(K_{\pi}+E_i\right)\cdot E_i  = -2$
for every exceptional component $E_j$, $j=1,\dots,s$. We say that $X$ is log-canonical (resp. log-terminal)
at $O$ if $k_j\geqslant -1$ ( resp. $k_j > -1$)  $ \forall j$.

\vskip 2mm

Associated to any point $\cn\in \R_{\geqslant0}^r$, we consider:

\vskip 2mm

$\cdot$ The {\it region} of $\cn $: \hskip 21mm
$\mathcal{R}_{\fab}\left(\cn\right)=\left\{\cn' \in
\R_{\geqslant 0}^r \hskip 2mm \left|
 \hskip 2mm \J\left(\fab^{\cn'}\right)\supseteq\J\left(\fab^{\cn}\right)\right\}\,\right.$

 \vskip 2mm

$\cdot$ The {\it constancy region} of $\cn$: \hskip 3mm
$\mathcal{C}_{\fab}\left(\cn\right)=\left\{\cn'  \in
\R_{\geqslant 0}^r \hskip 2mm \left| \hskip 2mm
\J\left(\fab^{\cn'}\right)=\J\left(\fab^{\cn}\right)\right\}\,\right.$

\vskip 2mm

\noindent The boundary of the region $\cR_{\fab}(\cn)$ is what we call  the {\it jumping wall} associated to $\cn$. One usually
refers to the jumping wall of the origin as the {\it log-canonical wall}. It follows from the definition of mixed multiplier ideals that the
jumping walls must lie on {\it supporting hyperplanes} of the form
\begin{equation} \label{eq:hyperplanes}
 V_{j,\ell}: \hskip 1mm e_{1,j} z_1+ \cdots +  e_{r,j} z_r= \ell + k_j \,,
 \hskip 8mm j=1,\dots,s
\end{equation} for a suitable $\ell \in \bZ_{>0}$. Here we assume that the effective divisors $F_i$ such that
$\fa_i\cdot\cO_{X'} = \cO_{X'}\left(-F_i\right)$, for  $ i=1,\dots,r$, are of the form  $F_i=\sum_{j=1}^s e_{i,j} E_j$.
Notice that each supporting hyperplane $V_{j,\ell}$ is associated to an exceptional component $E_j$. Indeed, we may find other
exceptional components associated to the same hyperplane, that is, we may find  $E_i$  and $\ell' \in \bZ_{>0}$ such that
$V_{j,\ell} = V_{i,\ell'}$.

\vskip 2mm

It is proved in  \cite[Theorem 3.3]{ACAMDC16}  that
the region $\cR_{\fab}(\cn)$ is  (the interior of) a {\it rational convex polytope} defined  by  the inequalities
\[e_{1,j} z_1 + \cdots +  e_{r,j} z_r < k_j + 1 + e_j^{{\cn}}\,,  \hskip 8mm j=1,\dots,s\]
corresponding to either rupture or dicritical divisors $E_j$ and  $D_{{\cn}}=\sum e_j^{{\cn}} E_j$ is the
{ antinef closure} of $\left\lfloor c_1 F_1 + \cdots +
c_rF_r  - K_{\pi} \right\rfloor$. 

The intersection of the boundary of a connected component of a constancy region $\cC_{\fab}(\cn)$ with a supporting hyperplane of $\cR_{\fab}(\cn)$ is what we call  a $\cC$-facet of $\cC_{\fab}(\cn)$. Every facet of a jumping wall decomposes into
several $\cC$-facets associated to different mixed multiplier ideals. 
%
%
%

\vskip 2mm

The main result of \cite{ACAMDC16} is an algorithm to compute all the constancy regions, and their corresponding mixed multiplier ideals, in any desired range of the positive orthant $\R_{\geqslant0}^r$. In particular the set of jumping walls of $\fab$, that we will
denote from now on as {\rm $ {\bf JW}_{\fab}$}, is precisely described.
The points on the jumping walls, 
which we will denote with  $\lambdab$ when we want to emphasize this fact, satisfy the property  $\J\left(\fab^{\cn}\right)\varsupsetneq\J\left(\fab^{\lambdab}\right)$
for all $\cn \in \{\lambdab - \R_{\geqslant0}^r\} \cap B_{\varepsilon}(\lambdab)$ and $\varepsilon >0$ small enough.
In the sequel, we will refer to these points as  the {\it jumping points} of  the tuple of ideals $\fab$.



\section{Multiplicities of jumping points} \label{multiplicities}
In this section we are going to provide a systematic study of the
multiplicity of any point $\cn\in\R^r_{\geqslant0}$. The results
that we present are a natural generalization of the ones we
obtained in  \cite{ACAMDCGA13}.

\begin{definition}
Let $\fab=\left(\fa_1,\ldots,\fa_r\right)\subseteq (\Oc_{X,O})^r$ be a tuple of $\fM$-primary ideals. We define the
multiplicity attached to a point $\cn\in\R^r_{\geqslant0}$ as the codimension
of $\J \left(\fab^{\cn}\right)$ in $\J \left(\fab^{(1-\varepsilon)\cn}\right)$ for
$\varepsilon>0$ small enough, i.e.

\[m(\cn):= \dim_{\bC} \frac{\J \left(\fab^{(1-\varepsilon)\cn}\right)}{\J \left(\fab^{\cn}\right)}\,.\]
\end{definition}


Our first goal is to compute explicitly these multiplicities using the theory of  {\it jumping divisors}  in this mixed multiplier ideals
setting as considered in \cite{ACAMDC16}. Since we are dealing with any general point, it will be more convenient
to consider  the  notion of {\it maximal jumping divisors} as opposed to the minimal jumping divisors, which are only meaningful for
jumping points.

\begin{definition}
Let $\fab=\left(\fa_1,\ldots,\fa_r\right)\subseteq
\left(\cO_{X,O}\right)^r$ be a tuple of ideals. Given any point $\cn \in \R^r_{\geqslant0}$, we define its {\it maximal jumping divisor} as the reduced divisor $H_{\cn} \leqslant
\sum_{i=1}^r F_i$ supported on those components $E_j$ such that
\[c_1 e_{1,j}  + \cdots +  c_ r e_{r,j} - k_j \in \bZ_{>0}\,.\]
Equivalently, for a sufficiently small $\varepsilon>0$,
\[H_{\cn}= \lceil K_\pi- (1-\varepsilon)c_1 F_1-\cdots- (1-\varepsilon)c_r  F_r\rceil -
\lceil K_\pi- c_1F_1-\cdots-c_rF_r\rceil\,.\]
In particular, we have
\[\J({\fab}^{(1-\varepsilon)\cn })= \pi_{*}\Oc_{X'}(\lceil K_\pi-  c_1 F_1 - \cdots - c_r F_r \rceil + H_{\cn})\,.\]
\end{definition}

\vskip 2mm

The following numerical properties for  maximal jumping divisors will be useful for our purposes.   We will skip the details of the proof just because
it is a natural generalization of \cite[Proposition 3.6]{ACAMDCGA13} and  the same proof holds mutatis mutandi.

\begin{proposition} \label{prop:Num_condition_cMMI}
Fix any $\pmb{c}\in \bR_{\geqslant0}^r$, and let $H_{\pmb{c}}$ be its associated maximal jumping divisor. Then the following inequalities hold:
\begin{itemize}
\item $\left(\left\lceil K_{\pi}- c_1 F_1-\cdots- c_r F_r\right\rceil + H_{\pmb{c}}\right)\cdot E_i \geqslant -1$ for all $E_i\leqslant H_{\pmb{c}}$, and
\item $\left(\left\lceil K_{\pi}- c_1 F_1-\cdots- c_r F_r\right\rceil + H_{\pmb{c}}\right)\cdot H \geqslant -1$ for any connected component $H\leqslant H_{\pmb{c}}$.
\end{itemize}
\end{proposition}

The main result of this section is the following:

\begin{theorem}\label{thm:multiplicity1_cMMI}
Let $\fab=\left(\fa_1,\ldots,\fa_r\right) \subseteq \left(\cO_{X,O}\right)^r$ be a tuple of $\fM$-primary ideals and $H_{\cn}$
the maximal jumping divisor associated to some $\cn \in \bR_{\geqslant0}^r$. Then,
\begin{align*}
m\left(\cn\right) & = \left(\left\lceil K_{\pi}-c_1F_1-\cdots-c_rF_r\right\rceil +
H_\cn\right)\cdot H_\cn +\#\left\{\text{connected components of }
H_\cn\right\}.
\end{align*}
\end{theorem}


\begin{proof}
Given the short exact sequence
\begin{align*}
 &0 \longrightarrow \Oc_{X'}\left(\left\lceil K_\pi-c_1F_1-\cdots -c_rF_r\right\rceil\right)
\longrightarrow \Oc_{X'}\left(\left\lceil K_\pi-c_1F_1-\cdots -c_rF_r\right\rceil + H_{\cn}\right)\longrightarrow\\
 &\qquad\qquad \longrightarrow \Oc_{H_{\cn}}\left(\left\lceil K_\pi-c_1F_1-\cdots -c_rF_r\right\rceil+H_{\cn}\right) \longrightarrow 0
\end{align*}
we have, after pushing it forward to $X$ and applying local vanishing \cite{Laz04} for the case of mixed multiplier ideals
\begin{multline*}
0 \longrightarrow\J(\fab^\cn) \longrightarrow\J(\fab^{(1-\varepsilon) \cn})\longrightarrow\\ \longrightarrow H^0\left(H_{\cn},\Oc_{H_{\cn}}\left(\left\lceil K_\pi-c_1F_1-\cdots -c_rF_r\right\rceil+H_{\cn}\right)\right) \otimes \bC_O \longrightarrow 0
\end{multline*}
for $\varepsilon$ small enough.   Therefore the multiplicity of $\cn$ is just
\begin{align*}
m\left(\cn\right) & = h^0\left(H_{\cn},\Oc_{H_{\cn}}\left(\left\lceil K_\pi-c_1F_1-\cdots -c_rF_r\right\rceil+H_{\cn}\right)\right) \\
& = \sum_{E_i \leqslant H_{\cn}} h^0\left(E_i,\Oc_{E_i}\left(\left\lceil
K_\pi-c_1F_1-\cdots -c_rF_r\right\rceil+H_{\cn}\right)\right) -  a_{H_{\cn}},
\end{align*}
where in the second equality we have used that $H_{\cn}$ has simple
normal crossings, and hence the sections of the line bundle
$\Oc_{H_{\cn}}\left(\left\lceil K_\pi-c_1F_1-\cdots -c_rF_r\right\rceil+H_{\cn}\right)$
correspond to sections over each component that agree on the
$a_{H_{\cn}}$ intersections, where $a_{H_{\cn}}$ denotes the number of edges  of $H_{\cn}$ in the dual graph.
Then, since we have
\begin{align*}
 &\deg \Oc_{E_i}\left(\left\lceil K_\pi-c_1F_1-\cdots -c_rF_r\right\rceil+H_{\cn}\right) =\left(\left\lceil K_\pi-c_1F_1-\cdots -c_rF_r\right\rceil+H_{\cn}\right)\cdot E_i \geqslant -1
\end{align*}
by Proposition \ref{prop:Num_condition_cMMI}, we get
\begin{align*}
m\left(\cn\right) & = \sum_{E_i \leqslant H_{\cn}}\left(\left(\left\lceil K_\pi-c_1F_1-\cdots -c_rF_r\right\rceil+H_{\cn}\right)\cdot E_i + 1\right) -  a_{H_{\cn}} \\
& = \left(\left\lceil K_\pi-c_1F_1-\cdots -c_rF_r\right\rceil+H_{\cn}\right)\cdot H_{\cn} + v_{H_{\cn}} -  a_{H_{\cn}} \\
& = \left(\left\lceil K_\pi-c_1F_1-\cdots -c_rF_r\right\rceil+H_{\cn}\right)\cdot H_{\cn} +\#\left\{\text{connected components of } H_{\cn}\right\}.
\end{align*}
\end{proof}

The above formula can be rephrased as follows


\begin{corollary} \label{cor:multiplicity2MMI}
Let $\fab=\left(\fa_1,\ldots,\fa_r\right) \subseteq \left(\cO_{X,O}\right)^r$ be a tuple of $\fM$-primary ideals and $H_\cn$
the maximal jumping divisor associated to some $\cn \in \bR_{\geqslant0}^r$. Then,
\begin{align*}
&m\left(\cn\right)  = \sum_{E_i \leqslant H_\cn} \left( \sum_{ E_j \in
\adj(E_i)} \left\{c_1 e_{1,j}+ \cdots+c_r e_{r,j} -
k_j\right\}+c_1\rho_{1,i}+\cdots +c_r\rho_{r,i}\right)\\&\qquad\qquad-\#\left\{\text{connected components of }
H_{\cn}\right\}.
\end{align*}
\end{corollary}

We may also provide a very simple numerical criterion to
detect whether a given point $\cn \in \bR^r_{>0}$ is a jumping point.

\begin{theorem} \label{cond_suf}
Let $\fab=\left(\fa_1,\ldots,\fa_r\right)\subseteq \left(\cO_{X,O}\right)^r$ be a tuple of $\fM$-primary ideals and $\cn \in
\bR^r_{\geqslant0}$. Then, $\cn$ is a jumping point if and only if $m(\cn) > 0$ or equivalently,  there exists a connected  component $H\leqslant
H_{\cn}$ such that
$$\left(\left\lceil K_\pi-c_1F_1-\cdots -c_rF_r\right\rceil + H_{\cn}\right)\cdot H \geq 0.$$ 
\end{theorem}

\begin{proof}
We have
\begin{align*}
m\left(\cn\right) & = \left(\left\lceil K_{\pi}-c_1F_1-\cdots-c_rF_r\right\rceil +
H_\cn\right)\cdot H_\cn +\#\left\{\text{connected components of }
H_\cn\right\} \\  & = \sum_{H\leqslant
H_{\cn}} \left( \left(\left\lceil K_{\pi}-c_1F_1-\cdots-c_rF_r\right\rceil + H_{\cn}\right)\cdot H +1\right),
\end{align*}
where the sum is taken over all the connected components $H\leqslant
H_{\cn}$. The result follows since  we have $\left(\left\lceil K_{\pi}-c_1F_1-\cdots-c_rF_r\right\rceil + H_c\right)\cdot H \geqslant -1$
by Proposition \ref{prop:Num_condition_cMMI}.
\end{proof}

\subsection{Poincar\'e series of mixed multiplier ideals}
Given a $\m$-primary ideal $\fa \subseteq \cO_{X,O}$ we consider its {\em Poincar\'e series}
\begin{equation} \label{eq-def-Poinc11}
P_\A (t)= \sum_{c\in \bR_{\geqslant0}} m(c) \hskip 1mm t^{c}.
\end{equation} which was first considered, in the case that $X$ is smooth and $\fa$ is simple,
by Galindo and Montserrat \cite{GM10} and extended in  \cite{ACAMDCGA13} to the case where $X$ has a rational singularity
and $\fa$ is any $\m$-primary  ideal.

\vskip 2mm

For a tuple of $\m$-primary ideals $\fab=\{\fa_1,\dots,\fa_r\} \subseteq \left(\cO_{X,O}\right)^r$ we are going
to give a generalization of this series by considering a sequence of mixed multiplier ideals indexed by points in a ray
$L: \cn_0 + \mu {\bf u}$ in the positive orthant $\bR^r_{\geqslant0} $ with  ${\bf u}=(u_1,\dots , u_r) \in \bZ^r_{\geqslant 0}$, ${\bf u}\neq {\bf 0}$ and $\cn_0\in\Q^r_{\geqslant0}$.  Here we are considering, for simplicity, a point $\cn_0$
belonging to a coordinate hyperplane but not necessarily being the origin and $\mu \in \bR_{\geqslant0}$. Namely, we consider the sequence of mixed multiplier ideals
$$\J \left(\fab^{\cn_0}\right) \supsetneq \J \left(\fab^{\cn_1}\right) \supsetneq \J \left(\fab^{\cn_2}\right)  \supsetneq  \cdots  \supsetneq  \J \left(\fab^{\cn_i}\right) \supsetneq  \cdots $$ where $\{{\cn_i}\}_{i>0}=L \cap {\rm \bf JW}_{\fab} $
 or equivalently $\{{\cn_i}\}_{i>0}$ is the set of jumping points of this sequence. Then we define the
 {\it Poincar\'e series of $\fab$ alongside the ray $L$} as

\begin{equation} \label{eq-def-Poinc22}
P_{\fab} (\underline{t}; L)= \sum_{\cn \in  L} m(\cn) \hskip 1mm \underline{t}^{\cn}.
\end{equation}
where $\underline{t}^{\cn}:=t_1^{c_1}\cdots t_r^{c_r}$.  Notice that we have
$$P_{\fab} (\underline{t}; L)= \sum_{\cn \in  L} m(\cn) \hskip 1mm \underline{t}^{\cn} =
\sum_{\cn \in [\cn_0,\cn_0 +{\bf u})} \sum_{k\in \bN} m(\cn+k {\bf u}) \hskip 1mm \underline{t}^{\cn +k{\bf u}}
=\underline{t}^{\cn_0}\sum_{\mu  \in [0,1)} \sum_{k\in \bN} m(\cn+k {\bf u}) \hskip 1mm \underline{t}^{(\mu +k){\bf u}}
$$ where the last equality follows from the fact that we are considering points of the form  $\cn=\cn_0+\mu {\bf u}$ with $\mu\in [0,1)$.
Our goal is to prove that this Poincar\'e series  is rational
in the sense that it belongs to the field of fractional functions $\bC(z_1,\dots , z_r)$, where the indeterminate $z_i$
corresponds to a fractional power $t_i^{1/e}$ for $e\in \bN_{>0}$ being the least common multiple of
the denominators of the coordinates of all jumping points.  To do so  we need to prove a linear recurrence among the
coefficients of the series. A key ingredient will be a periodicity property  of the maximal jumping divisor which
 follows from its definition.

\begin{lemma}\label{lem:sameMJD}
For any $\cn \in \bR^r_{>0}$ and $\alphan=(\alpha_1,\ldots,\alpha_r)\in\Z^r_{\geqslant0}$ we have $H_{\cn}= H_{\cn + \alphan}$.
\end{lemma}



The linear recurrence that the multiplicities satisfy is described in terms of the excesses at dicritical components.

\begin{proposition} \label{prop:growthMMI}
Let $\fab=\left(\fa_1,\ldots,\fa_r\right)\subseteq \left(\cO_{X,O}\right)^r$ be a tuple of $\fM$-primary ideals and
$\alphan=(\alpha_1,\ldots,\alpha_r)\in\Z^r_{\geqslant0}$.    
Then,
$$m\left(\cn+\alphan\right)- m\left(\cn\right) = \sum_{E_i \leqslant H_{\cn}} \sum_{j=1}^{r} \alpha_j\rho_{j,i}\,.$$
\end{proposition}
 \begin{proof}
Lemma \ref{lem:sameMJD} above states that $\cn$ and $\cn + \alphan$ have the same maximal jumping divisor, say $H_{\cn}$. Therefore, by Theorem \ref{thm:multiplicity1_cMMI}, we have
 $$m\left(\cn+\alphan\right)- m\left(\cn\right)= -(\alpha_1F_1+\cdots +\alpha_rF_r)\cdot H_{\cn}= \sum_{E_i \leqslant H_{\cn}} \sum_{1\leqslant j\leqslant r} \alpha_j\rho_{j,i}.$$
 \end{proof}

In the sequel, we will just denote
$\rho_{\cn, \alphan}:= \sum_{E_i \leqslant H_{\cn}} \sum_{j=1}^{r} \alpha_j\rho_{j,i}\,.$ The formula for the Poincar\'e series that we obtain  is the following:

\begin{theorem} \label{thm:rational}
 Let $\fab=\{\fa_1,\dots,\fa_r\} \subseteq \left(\cO_{X,O}\right)^r$ be a tuple of $\m$-primary ideals and let
$L: \cn_0 + \mu {\bf u}$  be a ray in the positive orthant $\bR^r_{\geqslant0} $ with ${\bf u}\in \bZ_{\geqslant 0}$, ${\bf u}\neq {\bf 0}$. The Poincar\'e series of $\fab$ alongside $L$ can be expressed as
$$P_{\fab} (\underline{t}; L)=  \underline{t}^{\cn_0} \sum_{\mu  \in [0,1)} \left( \frac{m(\cn_0+\mu{\bf u})}{1-\underline{t}^{\bf u}} + \rho_{\cn_0+\mu{\bf u}, {\bf u}}
\frac{\underline{t}^{\bf u}}{(1-\underline{t}^{\bf u})^2} \right) \underline{t}^{\mu {\bf u }} $$

\end{theorem}

\begin{proof}
Given a point  $\cn= \cn_0 + \mu {\bf u}$, with $\mu \in [0,1)$ we have, using Proposition \ref{prop:growthMMI}, that
\[m(\cn+k{\bf u})= m(\cn) + k\sum_{E_i\leqslant H_{\cn}}\sum_{j=1}^{ r}u_j\rho_{j,i} = m(\cn) + k \rho_{\cn, {\bf u}} \]

Therefore
\begin{eqnarray*}
\sum_{k\geqslant 0}m(\cn+k {\bf u})\hskip 1mm \underline{t}^{\cn+k {\bf u}} & = & m(\cn ) \hskip 1mm  \underline{t}^{\cn}+
(m(\cn) +  \rho_{\cn, {\bf u}}) \hskip 1mm \underline{t}^{\cn+ {\bf u}} + (m(\cn) +  2\rho_{\cn, {\bf u}}) \hskip 1mm  \underline{t}^{\cn+2 {\bf u}} + \cdots   \\
&= &  \left(\frac{m(\cn)}{1-\underline{t}^{\bf u}}+ \rho_{\cn, {\bf u}}\frac{\underline{t}^{\bf u}}{(1-\underline{t}^{\bf u})^2}\right) \underline{t}^{\cn}
\end{eqnarray*}
and the result follows.
\end{proof}


\begin{remark}
In the case that $ L $ is the $i$-th axis of  the positive orthant $\bR^r_{\geqslant0} $, in particular if $\cn_0$ is the origin,  we obtain the Poincar\'e series   of the ideal $\fa_i$.
\end{remark}

\section{Multiplicities of jumping points after small perturbations} \label{perturbations}

Let $\fab=\left(\fa_1,\ldots,\fa_r\right) \subseteq \left(\cO_{X,O}\right)^r$ be a tuple of $\m$-primary ideals
and consider two parallel rays  $L: \cn_0 + \mu {\bf u}$ and $L': \cn'_0 + \mu {\bf u}$ as those considered in the previous section
that are close enough. Our aim is to compare the sequences of mixed multiplier ideals indexed by points in both rays and see how
the multiplicity of a jumping point varies with a small perturbation. To illustrate this phenomenon we start with
the following example.

\begin{example} \label{rays}
Consider the tuple of ideals $\fab=(\fa_1,\fa_2)$ on a smooth surface $X$ given by:

\vskip 2mm

$\cdot$ $\fa_1=\left((x+y)^4,\right.$ $x^9(x+y),$ $x^{11},$ $x^6(x+y)^2,$ $\left.x^3(x+y)^3\right)$,

$\cdot$ $\fa_2=\left(y^3,x^7,x^5y,x^3y^2\right)$.

\vskip 2mm

\noindent The dual graph of the log-resolution of $\fab$ is as follows:

\vskip 2mm

\begin{center}
\begin{tabular}{c}
   \begin{tikzpicture}
   \draw  (-5,0) -- (4,0);
   \draw [dashed,->,thick,red] (2,0) -- (3,1);
   \draw [dashed,->,thick] (-4,0) -- (-5,1);
   \draw (-0.2,-0.3) node {{ $E_1$}};
   \draw (0.8,-0.3) node {{ $E_2$}};
   \draw (3.8,-0.3) node {{ $E_3$}};
   \draw (2.8,-0.3) node {{ $E_4$}};
   \draw (1.8,-0.3) node {{ $E_5$}};
   \draw (-1.2,-0.3) node {{ $E_{6}$}};
   \draw (-5.2,-0.3) node {{ $E_{7}$}};
   \draw (-2.2,-0.3) node {{ $E_{8}$}};
   \draw (-3.2,-0.3) node {{ $E_{9}$}};
   \draw (-4.2,-0.3) node {{ $E_{10}$}};
   \filldraw  (0,0) circle (2pt)
              (1,0) circle (2pt)
              (3,0) circle (2pt)
              (4,0) circle (2pt)
              (-1,0) circle (2pt)
              (-2,0) circle (2pt)
              (-3,0) circle (2pt)
              (-5,0) circle (2pt);
   \filldraw  [fill=white]  (2,0) circle (3pt)
                            (-4,0) circle (3pt);
   \end{tikzpicture}
\\
\end{tabular}
\end{center}

\noindent where the blank dots correspond to dicritical divisors and their excesses are represented by broken arrows.
The divisors associated with this resolution are

\vskip 2mm

$\cdot$ $F_1=4E_1+4E_2+4E_3+8E_4+12E_5+8E_6+11E_7+20 E_8+32E_9+44E_{10}$,

$\cdot$ $F_2=3E_1+6E_2+7E_3+14E_4+21E_5+3E_6+3E_7+6E_8+9E_9+12E_{10}$

\vskip 2mm

\noindent and the relative canonical divisor is:

\vskip 2mm

$\cdot$ $K_{\pi}=E_1+2E_2+3E_3+6E_4+9E_5+2E_6+3E_7+6E_8+10E_9+14E_{10}$.

\vskip 2mm

\noindent In Figure 2, we present the constancy regions of the corresponding mixed multiplier ideals, those regions 
are computed using the algorithm in \cite{ACAMDC16}. The chains of mixed multiplier ideals over the parallel rays $L: \left(0,\frac{101}{780}\right)  + \mu (1,1)$ and $L': \left(0,\frac{37}{390}\right) + \mu (1,1)$ are given in Table \ref{table:T1}. The sets of generators for these  ideals are computed using the
algorithm in \cite{AAB17} (see also \cite{BD17} ).



{ \begin{figure}[ht!!!]\label{fig:EX4.1}
  \begin{center}
  \medskip
    \includegraphics[width=.40\textwidth]{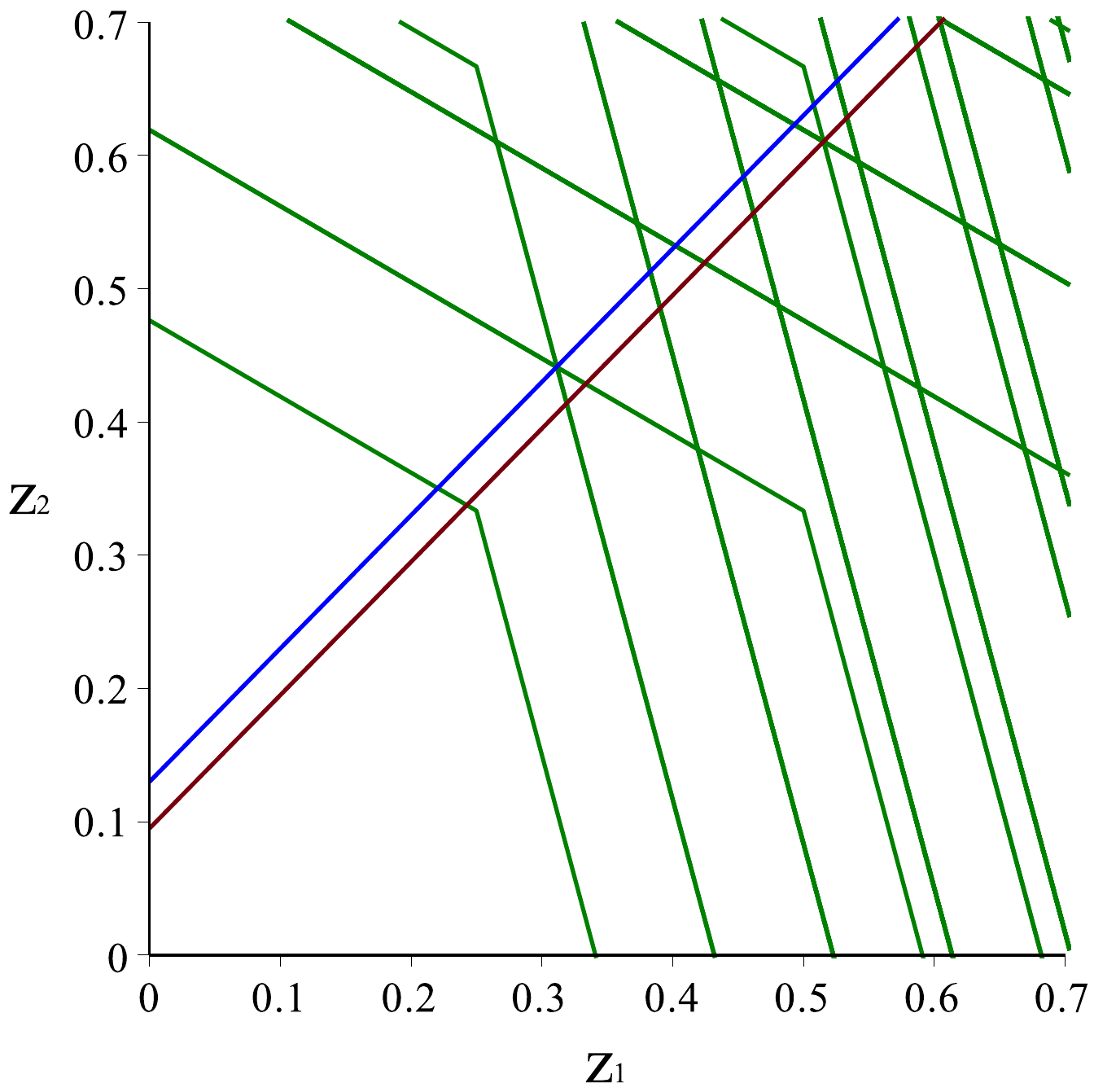}
  \end{center}
\caption{ Constancy regions of the mixed multiplier ideals of $\fab$ and the rays $L, L'$,  in blue and Bordeaux respectively.}
\end{figure}}

\def\arraystretch{1.3}
\begin{table}[!ht]
\centering
\begin{tabular}{|m{23mm}|m{49mm}|m{23mm}|m{49mm}|}
\hline
Jumping\newline points in $L$ &  \(\mathcal{J}(\fab^{\lambdab})\)& Jumping\newline points in $L'$&  \(\mathcal{J}(\fab^{\lambdab})\) \\
[1.5pt] \hline
\hline
$\left(\frac {81}{260}, \frac {86}{195} \right)$ 	& {\small $(x,y)$} & $\left(\frac {67}{130},\frac {119}{195} \right)$ 	 & {\small $(x,y)$} \\[1.5pt] \hline
\multicolumn{2}{| c |}{}& $\left(\frac {203}{520},\frac {757}{1560} \right)$ 	 & {\small $(x+y,x^2)$} \\[1.5pt] \hline
$\mathbf{ \left(\frac {631}{2860}, \frac {751}{2145} \right)}$ 	& {\small $(x^2,y^2,xy)$} & $\left(\frac {267}{455},\frac {1861}{2730} \right)$ 	 & {\small $(x^2,y^2,xy)$} \\[1.5pt] \hline
$\left(\frac {697}{1820}, \frac {1399}{2730} \right)$ 	 & {\small $(xy + x^2, x^3, y^2 + 2xy + x^2)$}& $\left(\frac {347}{1430},\frac {724}{2145} \right)$ 	 & {\small $(xy + x^2, x^3, y^2 + 2xy + x^2)$} \\[1.5pt] \hline
$\left(\frac {827}{1820}, \frac {797}{1365} \right)$ 	& {\small $(y^2 + xy, x^2y, x^2y + x^3)$}& $\left(\frac {477}{1430},\frac {919}{2145} \right)$ 	 & {\small $(y^2 + xy, x^2y, x^2y + x^3)$} \\[1.5pt] \hline
$\left(\frac {957}{1820}, \frac {1789}{2730} \right)$ 	& {\small $(y^2 + xy, x^2y, x^2y + x^3)$}& $\left(\frac {607}{1430},\frac {1114}{2145} \right)$ 	 & {\small $(y^2 + xy, x^2y, x^2y + x^3)$} \\[1.5pt] \hline
$\left(\frac {1151}{2860}, \frac {1141}{2145} \right)$ 	& {\small $( y^2 + xy, x^3y, x^3y + x^4)$}& \multicolumn{2}{ c |}{}  \\[1.5pt] \hline
$\left(\frac {1411}{2860}, \frac {1336}{2145} \right)$ 	& {\small $(y^3 + xy^2, x^3y, x^3y + x^4, x^2y^2,\newline xy^2 + x^2y)$} & $\mathbf{ \left(\frac {1161}{3640},\frac {4519}{10920} \right)}$ & {\small $(y^3 + xy^2, x^3y, x^3y + x^4, x^2y^2,\newline xy^2 + x^2y)$} \\[1.5pt] \hline
$\left(\frac {1849}{3640}, \frac {6961}{10920} \right)$ & {\small$(y^3 + 2xy^2 + x^2y, x^4y, xy^2 + x^2y,\newline x^2y^2 + 2x^3y + x^4, x^3y + x^4)$}& $\left(\frac {1681}{3640},\frac {6079}{10920} \right)$ & {\small$(y^3 + 2xy^2 + x^2y, x^4y, xy^2 + x^2y,\newline x^2y^2 + 2x^3y + x^4, x^3y + x^4)$} \\[1.5pt] \hline
$\left(\frac {2109}{3640},\frac {7741}{10920} \right)$ & {\small $( y^3 + 2xy^2 + x^2y, x^2y^2 + x^3y,\newline x^4y, x^2y^2 + 2x^3y + x^4)$} & $\left(\frac {1941}{3640},\frac {6859}{10920} \right)$ & {\small $( y^3 + 2xy^2 + x^2y, x^2y^2 + x^3y,\newline x^4y, x^2y^2 + 2x^3y + x^4)$} \\[1.5pt] \hline
$\left(\frac {2369}{3640},\frac {8521}{10920}\right)$ & {\small $( y^3 + 2xy^2 + x^2y, x^2y^2 + x^3y,\newline x^5y, x^2y^2 + 2x^3y + x^4)$}& $\left(\frac {2201}{3640},\frac {7639}{10920} \right)$ & {\small $( y^3 + 2xy^2 + x^2y, x^2y^2 + x^3y,\newline x^5y, x^2y^2 + 2x^3y + x^4)$} \\[1.5pt] \hline

\end{tabular}
\vspace{5pt}
\caption{  Chains of mixed multiplier ideals of $\fab$ over the rays $L$ and $L'$. In bold type we present the jumping points with multiplicity $2$.
} \label{table:T1}
\end{table}

\end{example}

In the previous example, we  observe that the chains of mixed multiplier ideals  differ whenever the
corresponding ray crosses the intersection of $\cC$-facets. Indeed, the multiplicity of a jumping point at the intersection of $\cC$-facets
is bigger than the multiplicities of jumping points in its neighborhood. The aim of this section is to provide an explanation to this phenomenon. We start with the fact that the multiplicity
does not increase in the interior of $\cC$-facets.

\begin{proposition}\label{interior}
Let $\fab=\left(\fa_1,\ldots,\fa_r\right)\subseteq \left(\cO_{X,O}\right)^r$ be a tuple of $\m$-primary ideals
and let $\lambdab, \lambdapb$ be two jumping points in the interior of a $\mathcal{C}$-facet.
Then $m(\lambdab)=m( \lambdapb)$.
\end{proposition}

To prove this result it is more convenient to compute the multiplicity
of a jumping point by using the so-called {\it minimal jumping divisor} instead of
the maximal jumping divisor as we did in Section \ref{multiplicities}. This minimal jumping
divisor is closely related to the algorithm developed in \cite{ACAMDC16} to compute
the constancy regions of mixed multiplier ideals.  We give its definition below but we refer
to \cite{ACAMDC16} for details.

\begin{definition} \label{def:minimal}
Let $\fab=\left(\fa_1,\ldots,\fa_r\right)\subseteq
\left(\cO_{X,O}\right)^r$ be a tuple of ideals. Given a jumping point $\lambdab=(\lambda_1,\dots , \lambda_r) \in \R^r_{\geqslant0}$, its corresponding
{\it minimal jumping divisor} is the reduced divisor $G_{\lambda} \leqslant
\sum_{i=1}^r F_i$ supported on those components $E_j$ for which the
point $\lambdab$ satisfies
$$ \lambda_1  e_{1,j}+ \cdots +   \lambda_r e_{r,j} = k_j + 1 +
e_j^{(1-\varepsilon){\lambdab}},$$ where, for a sufficiently small $\varepsilon>0$,
$D_{(1-\varepsilon)\lambdab} = \sum e_j^{(1-\varepsilon){\lambdab}}
\hskip 1mm E_j$ is the antinef closure of $$\lfloor
(1-\varepsilon){\lambda_1} F_1 + \cdots + (1-\varepsilon){\lambda_r}
F_r - K_\pi\rfloor.$$
\end{definition}

Using the same arguments that we used in the proof of Theorem \ref{thm:multiplicity1_cMMI}
we may provide the following formula for the multiplicity of a jumping point in terms of the minimal
jumping divisor.

\begin{proposition}\label{eq:multiplicity1_MMI}
Let $\fab=\left(\fa_1,\ldots,\fa_r\right) \subseteq \left(\cO_{X,O}\right)^r$ be a tuple of $\fM$-primary ideals and $G_{\lambdab}$
the maximal jumping divisor associated to some jumping point $\lambdab \in \bR_{>0}^r$. Then,
\begin{align*}
m(\lambdab) & = \left(\left\lceil K_\pi-\lambda_1 F_1-\cdots-\lambda_r F_r\right\rceil +
G_{\lambdab}\right)\cdot G_{\lambdab} +\#\{\text{connected components
of } G_{\lambdab}\}
\end{align*}
\end{proposition}

\vskip 2mm

It was proved in \cite[Lemma 4.6]{ACAMDC16} that two interior points of a $\mathcal{C}$-facet have the same minimal jumping divisor, which we refer to as {\it the minimal jumping divisor associated to} the $\mathcal{C}$-facet. Therefore by applying \ref{eq:multiplicity1_MMI}, the multiplicity is constant along the interior points of a $\mathcal{C}$-facet and thus proving Proposition \ref{interior}.
We point out that two interior points of a $\mathcal{C}$-facet may have  different maximal jumping divisor
%

\vskip 2mm

This constancy property for the multiplicities is no longer true when considering jumping  points at the intersection of $\mathcal{C}$-facets.
However we can control the multiplicity depending on the number of $\mathcal{C}$-facets that contain this jumping point.

\vskip 2mm

\begin{theorem} \label{intersection}

Let $\fab=\left(\fa_1,\ldots,\fa_r\right) \subseteq \left(\cO_{X,O}\right)^r$ be a tuple of $\m$-primary ideals
and let  $L: \cn_0 + \mu {\bf u}$ and $L': \cn'_0 + \mu {\bf u}$ be two parallel rays that are close enough.
Let $\lambdab \in L$ be a jumping point and  $B_\varepsilon(\lambdab)$ be a ball centered at $\lambdab$ of a sufficiently small radius $\varepsilon>0$ such that
$L\cap {\rm \bf JW}_{\fab} \cap B_\varepsilon(\lambdab) =\{\lambdab\}$. If
$L'\cap {\rm \bf JW}_{\fab} \cap B_\varepsilon(\lambdab) =\{\lambdab_1, \dots , \lambdab_n\}$ then
\[m(\lambdab)=m(\lambdab_1)+m(\lambdab_2)+\cdots+m(\lambdab_n)\,.\]

\end{theorem}

\begin{proof}

Let $V_1,\dots, V_k$ be all the hyperplanes associated to exceptional divisors that contain the jumping point $\lambdab$.
For each hyperplane $V_i$ we consider the divisor
$H_i = \sum_j  E_j$ where the sum is taken over the exceptional divisors that support the hyperplane $V_i$, i.e.  for all
$E_j \leq H_i$ there exists some $\ell_j \in \bZ_{>0}$ such that
the hyperplane $V_i$ is of the form $e_{1,j} z_1+ \cdots +  e_{r,j} z_r= \ell_j + k_j$.
Notice that, even though it is possible that not all of these hyperplanes support a jumping wall, we have a decomposition of the
maximal jumping divisor as $H_{\lambdab}=H_1 + \cdots + H_k$. Let $\{{\bf c}_1, \dots , {\bf c}_k\}$ be the ordered\footnote{The order
on the set of points $\{{\bf c}_1, \dots , {\bf c}_k\}$ is given by their distance to the origin. We order the hyperplanes $V_1,\dots, V_k$
accordingly.} set of points
resulting from the intersection of the ray $L'$ with the  hyperplanes $V_1,\dots, V_k$. Notice that we have
$L'\cap {\rm \bf JW}_{\fab} \cap B_\varepsilon(\lambdab) =\{\lambdab_1, \dots , \lambdab_n\} \subseteq
\{{\bf c}_1, \dots , {\bf c}_k\}$.

\vskip 2mm

For each ${\bf c}_i=({ c}_{i,1}, \dots , { c}_{i,r})$  we may find
a point $(1-\varepsilon'){\bf c}_i := ((1-\varepsilon'_1){c}_{i,1}, \dots , (1-\varepsilon'_r){ c}_{i,r})$ over the ray $L'$ that is close enough but smaller than ${\bf c}_1$ and a point  over the ray $L$ and smaller than $\lambdab $  that we will denote as
$(1-\varepsilon)\lambdab := ((1-\varepsilon_1)\lambda_1, \dots , (1-\varepsilon_r)\lambda_r)$ satisfying
$$\lceil K_{\pi} - (1-\varepsilon_1)\lambda_1 F_1 - \cdots - (1-\varepsilon_r)\lambda_rF_r\rceil = \lceil K_{\pi} - (1-\varepsilon'_1){ c}_{i,1} F_1 - \cdots - (1-\varepsilon'_r){ c}_{i,r}  F_r\rceil.$$

From the  construction of the hyperplanes $V_i$ we have:

\vskip 2mm

$\cdot$  $\lceil K_{\pi} - \lambda_1 F_1 - \cdots - \lambda_rF_r\rceil = \lceil K_{\pi} - (1-\varepsilon_1)\lambda_1 F_1 - \cdots - (1-\varepsilon_r)\lambda_rF_r\rceil + H_{\lambdab}.$

\vskip 2mm

$\cdot$ $\lceil K_{\pi} - { c}_{i,1} F_1 - \cdots - {c}_{i,r}  F_r\rceil = \lceil K_{\pi} - (1-\varepsilon'_1){c}_{i,1} F_1 - \cdots - (1-\varepsilon'_r){c}_{i,r}  F_r\rceil + H_1 + \cdots + H_i.$

\vskip 2mm

Therefore
\begin{equation} \label{epsilon}
\lceil K_{\pi} - \lambda_1 F_1 - \cdots - \lambda_rF_r\rceil - \lceil K_{\pi} - {c}_{i,1} F_1 - \cdots - { c}_{i,r}  F_r\rceil = H_{i+1} + \cdots H_k.
\end{equation}

\vskip 2mm

By Theorem \ref{thm:multiplicity1_cMMI}, one has
\begin{multline*}
m(\lambdab)  = \left(\left\lceil K_\pi-\lambda_1 F_1-\cdots-\lambda_r F_r\right\rceil +
H_{\lambdab}\right)\cdot H_{\lambdab} +\#\{\text{connected components
of } H_{\lambdab}\}\,.
\end{multline*}
Thus, we can rewrite this formula as
\begin{align*}
 m(\lambdab)&= \sum_{i=1}^k (\lceil K_{\pi} - \lambda_1 F_1 - \cdots - \lambda_rF_r\rceil+H_{i})\cdot H_i\\
 &\qquad+\sum_{i=1}^k\sum_{\substack{j=1 \\ j\neq i}}^k H_iH_j+\#\{\text{connected components of }H_{\lambdab}\} \\
 &= \sum_{i=1}^k(\lceil K_{\pi} - { c}_{i,1} F_1 - \cdots - { c}_{i,r}F_r\rceil+H_{i})\cdot H_i\\
 &\qquad+\sum_{i=1}^k\sum_{j>i}^k H_iH_j+\#\{\text{connected components of }H_{\lambdab}\}\,
\end{align*}

\noindent where the last equality follows from Equation \ref{epsilon}.  Now, recall that for any divisor $D$ with exceptional support
$\#\{\text{connected components of }D\}=v_{D}-a_{D},$
where $v_{D}$ and $a_{D}$ denote the number of vertices and edges of $D$
in the dual graph. Since $v_{H_{\lambdab}}=v_{H_1}+\cdots+v_{H_k}$ and $a_{H_{\lambdab}}=a_{H_1}+\cdots+a_{H_k}+\sum_{i=1}^k\sum_{j>i}^k H_iH_j$ we deduce
\begin{align*}
     \#\{\text{\small{connected components of }}H_{\lambdab}\}
      =\sum_{i=1}^k\#\{\text{\small{connected components of }}H_i\}
      -\sum_{i=1}^k\sum_{j>i}^k H_iH_j .
\end{align*}
Therefore
\begin{align*}
m(\lambdab) & = \sum_{i=1}^k \left[ (\lceil K_{\pi} - {c}_{i,1} F_1 - \cdots - {c}_{i,r}F_r\rceil+H_{i})\cdot H_i  + \#\{\text{connected components of }H_i\} \right] \\
&= m({\bf c}_1)+ \cdots + m({\bf c}_k) .
\end{align*}

The only points with non zero multiplicity are those over a jumping wall, namely the jumping points $\{\lambdab_1, \dots , \lambdab_n\}$
and thus we get the desired result.
\end{proof}

\section{Contribution to the log-canonical wall} \label{Sec4}

Let $X$ be a smooth complex surface and  $\fa\subseteq \Oc_{X,O}$ an ideal. A common theme in the
study of multiplier ideals is to check which exceptional divisors contribute to the jumping numbers of $\fa$.
In the case of the  log-canonical threshold we know that it is  described by the formula
$${\rm lct}(\fa)= \min_i\left\{\frac{k_i+1}{e_i}\right\}.$$
In the case that $\fa$ is $\m$-primary and simple, this minimum is achieved at  the first
rupture or dicritical  exceptional component, starting from the origin, in the dual graph of the log-resolution of $\fa$
(see \cite{Jar11}, \cite{Tuc10}). For non simple ideals we may find
some analogous statements in  \cite{Kuw99}, \cite{GHM12}, \cite{ACNLM08}, \cite{AN10}.

\vskip 2mm

For the case of mixed multiplier ideals, Cassou-Nogu\`es and Libgober \cite[Theorem 4.22]{CNL14} studied the contribution
of exceptional divisors to the log-canonical wall for the case where the tuple of ideals corresponds to the branches
of a plane curve.
%
%
In this section we will give a generalization of their result that works for general
tuples of $\fM$-primary ideals  $\fab=\left(\fa_1,\ldots,\fa_r \right)\subseteq \left(\Oc_{X,O}\right)^r$,
where $X$ is a complex surface with a rational singularity at $O$ and the points in the log-canonical wall
have multiplicity one.

\vskip 2mm

Their result is described in terms of the so-called {\it Newton nest} introduced in \cite[Definition 4.19]{CNL14}. In order to give a generalization to our setup of the Newton nest we will need to fix some notation.
When $X$ has a rational singularity  we may have an strict inclusion
$\Oc_{X,O}\varsupsetneq\J(\fab^{\bf 0})$ where ${\bf 0}=(0,\dots, 0)$ is the origin of the positive orthant $\bR^r_{\geqslant 0}$.
Indeed,  the mixed multiplier ideal $\J(\fab^{\bf 0})$  is described by a divisor $D_{\bf 0}=\sum e_j^{\bf 0} E_j$
which is the antinef closure of $\lfloor  - K_\pi \rfloor$ that can be computed using the {\it unloading procedure}
described in \cite{ACAMDC13}. Therefore, the log-canonical wall is supported on hyperplanes of the form
\[  e_{1,j} z_1 + \cdots +  e_{r,j} z_r = k_j + 1 + e_j^{{\bf 0}}\,,  \hskip 8mm j=1,\dots,s.\]
For each point $\pmb{z}_i=(0,\ldots,0,{\rm lct}(\fa_i),0,\ldots,0)$  in the
$i$-th coordinate axis corresponding to the log-canonical threshold of the ideal $\fa_i$, $i=1,\dots, r$, we
consider the reduced divisor $G'_{\pmb{z}_i}= \sum E_j$,  where the sum is taken over those exceptional divisors
associated to the supporting hyperplanes  of the log-canonical wall which contain the point ${\pmb{z}_i}$. Notice that this divisor
is contained in the minimal jumping divisor of ${\pmb{z}_i}$, that is $G'_{\pmb{z}_i}\leqslant G_{\pmb{z}_i}$.

\vskip 2mm

%

\begin{definition} \label{Newton}
 Consider the minimal connected subgraph $\Gamma_{\fab}'$ of the dual graph $\Gamma_{\fab}$
 containing the divisors $G'_{\pmb{z}_i}$, for $i=1\dots, r$. The {\it Newton nest} of $\Gamma_{\fab}$
 is the set of rupture or dicritical divisors belonging to $\Gamma'_{\fab}$.
\end{definition}

\begin{remark}
 In the case that $X$ is smooth and the ideals $\fab_i$  are simple,  this definition coincides with the one
given  by Cassou-Nogu\`es and Libgober in \cite[Definition 4.19]{CNL14} since in this case
we have that $G'_{\pmb{z}_i}= G_{\pmb{z}_i}=E_{j_i}$, where $E_{j_i}$ is the rupture divisor in the dual graph
$\Gamma_{\fa_i}$ which is closest to its root.
 \end{remark}

\vskip 2mm

Cassou-Nogu\`es and Libgober  \cite[Theorem 4.22]{CNL14} established a one-to-one correspondence between the divisors of the Newton nest and the $\cC$-facets of the log-canonical wall in the case where $X$ is smooth and the tuple of ideals correspond to the branches of a plane curve. The only restriction that we are going to impose in our generalization is that the multiplicity of all the points in the log-canonical wall have multiplicity one. This condition is achieved, for example, in the case that $X$ has a log-terminal singularity at $O\in X$.

\begin{lemma}
Let
$\fab=\left(\fa_1,\ldots,\fa_r\right)\subseteq \left(\cO_{X,O}\right)^r$ be a tuple of simple $\fM$-primary ideals
and $X$ is a complex surface with a log-terminal singularity. Then, all the points in the log-canonical wall
have multiplicity one.
\end{lemma}

\begin{proof}
 From the definition of log-terminal singularity, it follows that the antinef closure of $\lfloor -K_{\pi}\rfloor$ is $0$
 because all the coefficients of $\lfloor -K_{\pi}\rfloor$ are strictly smaller than one. Therefore the ideal associated
 to the point ${\pmb 0}$ is the whole ring.

 \vskip 2mm

 Let $\lambdab$ be a jumping point in  the log-canonical wall. All the coefficients of the divisor
 $\lfloor \lambdab F -K_{\pi} \rfloor$ must be smaller or equal to one so we have
 $\lfloor \lambdab F -K_{\pi} \rfloor \leqslant Z$ where $Z$ is the fundamental cycle. Therefore we have
 $\fM=\pi_*\Oc_{X'}(-Z)\subseteq \J\left(\fab^{\lambdab}\right) \subsetneq \cO_{X,O}$. So $ \J\left(\fab^{\lambdab}\right)=\fM$,
 and consequently $m(\lambdab)= 1$ for all points in the log-canonical wall.
%
%
\end{proof}

\vskip 2mm

Before stating the main result of this section  we will present some properties concerning  jumping points of multiplicity one.
This is a very restrictive condition on the corresponding minimal jumping divisors. To such purpose we have to introduce
some technical notation. Given any exceptional component $E_i$ and a reduced  divisor
$D \leqslant E =Exc(\pi) $, we define the set of components adjacent to $E_i$ {\em inside} $D$ and its number as:
$$\adj_D\left(E_i\right) = \left\{E_j \leqslant D \hskip 2mm | \hskip 2mm E_i \cdot E_j = 1\right\} \quad \text{and} \quad a_D\left(E_i\right) = \#\adj_D\left(E_i\right)$$

\begin{lemma}\label{lem:1divisor}
 Let $\fab=\left(\fa_1,\ldots,\fa_r\right) \subseteq \left(\cO_{X,O}\right)^r$ be a tuple of $\fM$-primary ideals and $\lambdab$ a jumping point such that $m(\lambdab)=1$. Then, the minimal jumping divisor $G_{\lambdab}$  has only one connected component and no rupture or dicritical divisor $E_i$ such that $a_{G_{\lambdab}}(E_i)>1$.
\end{lemma}

\begin{proof}
Using Proposition \ref{eq:multiplicity1_MMI}  we have  the following formula.
$$
m(\lambdab)  = \left(\left\lceil K_\pi-\lambda_1 F_1-\cdots-\lambda_r F_r\right\rceil +
G_{\lambdab}\right)\cdot G_{\lambdab}+\#\{\text{connected components
of } G_{\lambdab}\}\,.
$$
In the case that $m(\lambdab)=1$
we can deduce that
$\#\{\text{connected components of } G_{\lambdab}\}=1$ and
\begin{align*}
      \left(\left\lceil K_\pi-\lambda_1 F_1-\cdots-\lambda_r F_r\right\rceil +
G_{\lambdab}\right)\cdot G_{\lambdab}=\sum_{E_i\leqslant G_{\lambdab}} \left(\left\lceil K_\pi-\lambda_1 F_1-\cdots-\lambda_r F_r\right\rceil +
G_{\lambdab}\right)\cdot E_i=0
\end{align*}
since we already had $\left(\left\lceil K_\pi-\lambda_1 F_1-\cdots-\lambda_r F_r\right\rceil +
G_{\lambdab}\right)\cdot G_{\lambdab}\geqslant 0$  by   \cite[Proposition 4.13]{ACAMDC16}.
Indeed, using again  this result we have
\[\left(\left\lceil K_\pi-\lambda_1 F_1-\cdots-\lambda_r F_r\right\rceil +
G_{\lambdab}\right)\cdot E_i=0\] for all $E_i\leqslant G_{\lambdab}$.
We may provide a more explicit description of this equation using  \cite[Lemma 4.11]{ACAMDC16}. Namely
we have
\begin{multline*}
\left(\left\lceil K_\pi-  \lambda_1  F_1-\cdots-\lambda_r F_r\right\rceil  + G_{\lambdab}\right)\cdot E_i  =  \\
 =  -2 + \lambda_1 \rho_{1,i}+\cdots+ \lambda_r \rho_{r,i} +a_{G_{\lambdab}}\left(E_i\right) + \sum_{ E_j \in \adj_E(E_i)} \left\{\lambda_1 e_{1,j}+\cdots+\lambda_r e_{r,j}-k_j\right\}  \,.
\end{multline*}
Thus, if $E_i$ is a rupture or dicritical component with $a_{G_{\lambdab}}\left(E_i\right)>1$, then we have \[\left(\left\lceil K_\pi-\lambda_1  F_1-\cdots-\lambda_r F_r\right\rceil + G_{\lambdab}\right)\cdot E_i>0\,\]  so we get a contradiction and the result follows.
\end{proof}

\begin{corollary}\label{cor:Cfacet}
Let $\lambdab \in \bR^r_{\geqslant 0}$ be a jumping point not contained in any coordinate hyperplane such that  $m(\lambdab)=1$. Then:

 \begin{itemize}
 \item[i)] If  $\lambdab$ is an interior  point of a $\cC$-facet which does not intersect any other $\cC$-facet,
  the minimal jumping divisor $G_{\lambdab}$ contains at most two dicritical or rupture divisors.
   \item[ii)] If  $\lambdab$ is an interior  point of a $\cC$-facet which intersects, at least, another $\cC$-facet,
  the minimal jumping divisor $G_{\lambdab}$ is a dicritical or rupture divisor.
  \item[iii)] If $\lambdab$ is at the intersection of two  $\cC$-facets,  the minimal jumping divisor $G_{\lambdab}$
  is connected and contains exactly two dicritical or rupture divisors, which are its two ends.

 \end{itemize}

\end{corollary}

\begin{proof}
Let $\lambdab \in \bR^r_{\geqslant 0}$ be a jumping point. By \cite[Theorem 4.14]{ACAMDC16},  the ends of the connected
components of the minimal jumping divisor $G_{\lambdab}$ over the dual graph are either rupture or dicritical divisors.
If we assume $m(\lambdab)=1$, then, using Lemma \ref{lem:1divisor}, we have that  $a_{G_{\lambdab}}(E_j)\leqslant 1$
for any rupture or dicritical divisor $E_j$. Therefore either $G_{\lambdab}$ is just one exceptional component or it
is connected with just two ends which are rupture or dicritical divisors in the dual graph.
In particular,  {\rm i)} follows.

\vskip 2mm

Now assume that $\lambdab$ is at the intersection of two $\cC$-facets $\cC_1$ and $\cC_2$
with associated minimal jumping divisors $G_1$ and $G_2$ respectively. We
have $G_{\lambdab}=G_1+G_2+G'$ for some divisor $G'$ with exceptional support. Moreover,
the jumping points in the interior of $\cC_1$ and $\cC_2$ have multiplicity $1$ so
the same properties considered above also apply for $G_1$ and $G_2$. The two
$\cC$-facets are supported on different hyperplanes with different slope, so  $G_1$ and $G_2$ do not share
any exceptional divisor. By Lemma \ref{lem:1divisor}, this forces $G_1$ and $G_2$ to be just one exceptional component
being a rupture or dicritical divisor and the minimal jumping divisor
$G_{\lambdab}$ contains exactly two dicritical or rupture divisors. Thus,  {\rm ii)} and {\rm iii)} follow.
\end{proof}


\begin{remark}
Given a tuple of $\fM$-primary ideals $\fab=(\fa_1,\ldots,\fa_r )\subseteq (\cO_{X,O})^r$ we may pick  a subfamily
$\fa'=\{\fa_{i_1}, \dots , \fa_{i_k} \hskip 2mm | \hskip 2mm 1\leq i_1 < \cdots < i_k \leq r  \}$ and, if no confusion arise, we may view
it either as a tuple $ (\cO_{X,O})^s$ or a subtuple  of $\fab$ in $(\cO_{X,O})^r$ in the obvious way. Notice for example that
the Newton nest of $\fab'$   is a subset of the Newton nest of $\fab$.
 In the case that $\lambdab=(\lambda_1, \dots , \lambda_r) \in \bR^r_{\geqslant 0}$ is a jumping point contained in a coordinate hyperplane, we may consider the tuple  $\fab'=(\fa_i \hskip 2mm | \hskip 2mm \lambda_i \neq 0)$. Corollary \ref{cor:Cfacet}
 holds whenever we consider $\lambdab$ as a jumping point for $\fab'$ and thus, a point not in the coordinate
 hyperplanes of the lower dimensional positive orthant.

\end{remark}

\vskip 2mm

Notice that Corollary \ref{cor:Cfacet} already singles out a very particular case where we may not have our desired one-to-one correspondence.
Namely, assume that the log-canonical wall has a unique $\cC$-facet with points of multiplicity one. Part ${\rm i)}$ of Corollary \ref{cor:Cfacet}
says that the Newton nest contains either one or two divisors. Therefore the desired one-to-one correspondence fails when we have exactly two divisors and this case can indeed be achieved. Recall that the effective divisors $F_i$ such that
$\fa_i\cdot\cO_{X'} = \cO_{X'}\left(-F_i\right)$  are of the form  $F_i=\sum_{j=1}^s e_{i,j} E_j$   for  $ i=1,\dots,r$ and the relative canonical divisor is $ K_{\pi}=\sum_{i=1}^s k_j E_j $. Let $E_j$ and $E_{\ell}$ be the divisors in the Newton nest and
$V_{j,1}:\hskip 1mm e_{1,j} z_1+ \cdots +  e_{r,j} z_r=  k_j +1 $ and  $  V_{\ell,1}: \hskip 1mm e_{1,\ell} z_1+ \cdots +  e_{r,\ell} z_r=  k_\ell +1$ be their associated hyperplanes. The numerical conditions for which these hyperplanes support the unique $\cC$-facet of the  log-canonical wall are 
\[\frac{e_{1,\ell}}{e_{1,j}}=\cdots=\frac{e_{r, \ell}}{e_{r,j}}=\frac{k_\ell+1}{k_j+1}\,.\]
This result can be reformulated in the following



\vskip 2mm

\begin{lemma}\label{lem:2rdiv}

Let
$\fab=\left(\fa_1,\ldots,\fa_r\right)\subseteq \left(\cO_{X,O}\right)^r$ be a tuple of $\fM$-primary ideals
where $X$ is a complex surface with a rational singularity at $O \in X$. Assume that all the points in the log-canonical wall have multiplicity one and the Newton nest contains two divisors $E_j$ and $E_{\ell}$.
Then, the log-canonical wall has a unique $\cC$-facet if and only if
\[\frac{e_{1,\ell}}{e_{1,j}}=\cdots=\frac{e_{r, \ell}}{e_{r,j}}=\frac{k_\ell+1}{k_j+1}\,.\]
%
\end{lemma}

\vskip 2mm

We illustrate this case with the following

\begin{example}

Consider a smooth surface $X$ and a tuple of ideals $\fab=\{\fa_1,\fa_2\}$ such that they have a minimal log-resolution with the following vertex ordering

\begin{center}

   \begin{tikzpicture}[scale=0.9]
   \draw  (-1,0) -- (3,0);
   \draw  (2,0) -- (3,1);
   \draw  (2,0) -- (3,-1);
   \draw  (3,1) -- (5,1);
   \draw  (3,-1) -- (5,-1);
   \draw  (-1,0) -- (-2,-1);
   \draw  (-1,0) -- (-2,1);
   \draw  (-2,-1) -- (-3,-1);
   \draw  (-2,1) -- (-4,1);
   \draw [dashed,->,thick,red] (4,1) -- (5,0);
   \draw [dashed,->,thick,blue] (3,-1) -- (4,0);
   \draw [dashed,->,thick,red] (-2,1) -- (-3,0.33);
   \draw [dashed,->,thick,blue] (-2,1) -- (-3,0);
   \draw (-.2,-0.3) node {{\small $E_1$}};
   \draw (2.8,-0.3) node {{\small $E_2$}};
   \draw (0.8,-0.3) node {{\small $E_3$}};
   \draw (1.8,-0.3) node {{\small $E_4$}};
   \draw (4.8,-1.3) node {{\small $E_5$}};
   \draw (3.8,-1.3) node {{\small $E_6$}};
   \draw (2.8,-1.3) node {{\small $E_7$}};
   \draw (4.8,1.3) node {{\small $E_8$}};
   \draw (2.8,1.3) node {{\small $E_9$}};
   \draw (3.8,1.3) node {{\small $E_{10}$}};
   \draw (-3.2,-1.3) node {{\small $E_{11}$}};
   \draw (-2.2,-1.3) node {{\small $E_{12}$}};
   \draw (-1.5,0) node {{\small $E_{13}$}};
   \draw (-4.2,1.3) node {{\small $E_{14}$}};
   \draw (-3.2,1.3) node {{\small $E_{15}$}};
   \draw (-2.2,1.3) node {{\small $E_{16}$}};
   \filldraw  (0,0) circle (2pt)
              (1,0) circle (2pt)
              (2,0) circle (2pt)
              (3,0) circle (2pt)
              (3,1) circle (2pt)
              (5,1) circle (2pt)
              (4,-1) circle (2pt)
              (5,-1) circle (2pt)
              (-1,0) circle (2pt)
              (-2,-1) circle (2pt)
              (-3,-1) circle (2pt)
              (-3,1) circle (2pt)
              (-4,1) circle (2pt);
   \filldraw  [fill=white]  (-2,1) circle (3pt)
                            (4,1) circle (3pt)
                            (3,-1) circle (3pt);
   \end{tikzpicture}

   \end{center}

\noindent and the divisors given by the ideals are:

\vskip 2mm

$\cdot$ $F_1= (18,24,45,72,73,146,218,72,144,216,21,42,63,64,128,194)$, and

$\cdot$ $F_2= (18,24,45,72,72,144,216,74,147,222,21,42,63,64,128,194)$.

\vskip 2mm

\noindent with $K_\pi=(1,2,4,7,8,16,24,8,16,25,2,4,6,7,14,21)$.

\vskip 2mm

The divisors in the Newton nest are $E_4$ and $E_{13}$ and the log-canonical wall only has a unique $\cC$-facet whose
supporting hyperplane has the following equation: \[ V_{4,1}: 72x_1+72x_2=8 \text{, or equivalently,  }  V_{13,1}: 63x_1+63x_2=7\,.\]
\end{example}

The main result of this section is that the one-to-one correspondence established by Cassou-Nogu\`es and Libgober  in \cite[Theorem 4.22]{CNL14} still holds in our setup except for this very particular case where we may have two divisors in the Newton nest and just a unique $\cC$-facet.

\begin{theorem} \label{thm:Pi} Let
$\fab=\left(\fa_1,\ldots,\fa_r\right)\subseteq \left(\cO_{X,O}\right)^r$ be a tuple of $\fM$-primary ideals
where $X$ is a complex surface with a rational singularity at $O \in X$. Assume that the log-canonical wall has at least two $\cC$-facets and all its points
have multiplicity one.
Then, there is a one-to-one correspondence between the exceptional divisors
 in the Newton nest of $\fab$  and the $\cC$-facets of the log-canonical wall.
\end{theorem}

\begin{proof}
We start with the case $r=2$, that is $\fab=\left(\fa_1,\fa_2\right)$.
%
%
In order to construct the Newton nest of $\fab$ we start considering
two points $\pmb{z}_1$ and $\pmb{z}_2$ in the coordinate axes corresponding to the log-canonical
thresholds of the ideals $\fa_1$ and $\fa_2$ respectively.
Then we order the $\cC$-facets of the log-canonical wall $\cC_1, \cC_2, \dots , \cC_l$
in such a way that $\pmb{z}_1 \in \cC_1$, $\pmb{z}_2 \in \cC_l$ and each $\cC$-facet $\cC_i$
intersects $\cC_{i-1}$ and $\cC_{i+1}$. Roughly speaking, we are considering a path from $\pmb{z}_1$ to $\pmb{z}_2$
over the log-canonical wall.
By Corollary \ref{cor:Cfacet}, the exceptional divisors associated to the supporting hyperplanes of
$\cC_1, \cC_2, \dots , \cC_l$ are unique, but they are also different because the hyperplanes
have different slopes. Therefore,
we can order these dicritical or rupture divisors
$E_{j_1}, E_{j_2}, \dots, E_{j_l}$ accordingly to their corresponding $\cC$-facets. Moreover,
Corollary \ref{cor:Cfacet}  implies that they form a path in the dual graph $\Gamma_{\fab}$ of the
log-resolution of $\fab$  with no dicritical or rupture divisors in between two consecutive $E_{j_i}$'s.

\vskip 2mm

Now, let $\Gamma_{\fab}' \subset \Gamma_{\fab}$ be the minimal connected subgraph containing $E_{j_1}$ and $E_{j_s}$.
Then we have that $E_{j_2}, \dots, E_{j_{s-1}}$  belong to $\Gamma_{\fab}'$
and the result follows because the Newton nest is the set  $\{E_{j_1}, E_{j_2}, \dots, E_{j_s}\}$ by construction.

\vskip 2mm

For the case $r>2$, that is $\fab=\left(\fa_1,\ldots,\fa_r\right)$, we  have to consider
points $\pmb{z}_1, \dots , \pmb{z}_r$ in the coordinate axes corresponding to the log-canonical
thresholds. We are going to pick one of them, say $\pmb{z}_1$, and a  $\cC$-facet $\cC_1$ containing this point $\pmb{z}_1$.
Let ${\bf c} \in \cC_1$ be an interior point with rational coordinates and ${\bf q}:=(0,q_2,\dots,q_r)$ a point in the coordinate hyperplane
$\{x_1=0 \}$ with rational coordinates. Notice that ${\bf q}$ corresponds to the mixed multiplier ideal
$\J(\fa_2^{q_2} \cdots \fa_r^{q_r})=\J\left((\fa_2^{dq_i}\cdots \fa_r^{dq_r})^{\frac{1}{d}}\right)$ for some $d\in\Z$ such that $dq_i\in\Z$ for $i=2,\ldots,r$.

\vskip 2mm

The mixed multiplier ideals appearing in the restriction of the positive orthant $\bR^r_{\geqslant 0}$
to the plane containing the points ${\bf c}$, ${\bf q}$ and the origin ${\bf 0}$  are the mixed
multiplier ideals of the duple $\fab'= ( \fa_1, \fa_2^{dq_2}\cdots \fa_r^{dq_r})$ so the facets of
the corresponding log-canonical wall are in one-to-one correspondence with the exceptional divisors
in the Newton nest of $\fab'$ which is contained in the Newton nest of $\fab$. Just moving the points
${\bf c}$ and ${\bf q}$ conveniently allows us to cover the whole log-canonical wall of $\fab$ and the result
follows. We point out that a log-resolution of $\fab$ is also a
log-resolution of $\fab'$.
\end{proof}

In the following example we show that with our definition of the Newton nest we may also
consider the case of non-simple ideals in a smooth surface which was not considered in \cite{CNL14}.

\begin{example} \label{ex:nest3}

Consider the tuple of ideals $\fab=(\fa_1,\fa_2,\fa_3)$  on a smooth surface $X$:

\vskip 2mm

$\cdot$ $\fa_1=( {y}^{3},{x}^{6}y,{x}^{8},{x}^{3}{y}^{2})$,

$\cdot$ $\fa_2=(x\left(x^2+x+y\right)^4,  \left( x^2+x+y \right)^2\left( x^2-x-y \right)^2, x^5 \left( x^2-x-y \right) ^2, x^5 \left( x^2+x+y \right) ^2,$

\hskip 1.3cm $x^3\left(x^2+x+y\right)\left( x^2-x-y \right)^2,$
$x^3 \left( x^2+x+y\right) ^3),$

$\cdot$ $\fa_3=($ $-x^6 \left( x-y \right),$ $x^7,$ $x^5
 \left( x-y \right) ^2,$ $-x^3 \left( x-y \right) ^3, x^
5 \left( x-y \right) ^2,$ $- \left( x-y \right) ^5,x^2 \left( x-y \right) ^4).$

\vskip 2mm

\noindent The dual graph of the log-resolution of $\fab$ is:

\hspace{1mm}
\begin{center}

\begin{tabular}{c}
\begin{tikzpicture}[scale=0.8]
   \draw  (0,0) -- (-1,0);
   \draw  (0,0) -- (1,1);
   \draw  (0,0) -- (1,-1);
   \draw  (4,1) -- (1,1);
   \draw  (4,-1) -- (1,-1);
   \draw  (-1,0) -- (-2,1);
   \draw  (-1,0) -- (-2,-1);
   \draw  (-2,1) -- (-3,1);
   \draw  (-2,-1) -- (-3,-1);
   \draw [dashed,->,thick,red] (3,1) -- (4,.5);
   \draw [dashed,->,thick,blue] (2,-1) -- (3,-.5);
   \draw [dashed,->,thick] (-2,-1) -- (-3,-.5);
   \draw [dashed,->,thick] (-2,1) -- (-3,.5);
   \draw (0.4,0) node {{ $E_1$}};
   \draw (0.7,1.4) node {{ $E_2$}};
   \draw (3.7,1.4) node {{ $E_3$}};
   \draw (1.7,1.4) node {{ $E_4$}};
   \draw (2.7,1.4) node {{ $E_5$}};
   \draw (-1.6,0) node {{ $E_{6}$}};
   \draw (-2.7,1.3) node {{ $E_{7}$}};
   \draw (-1.7,1.3) node {{ $E_{8}$}};
   \draw (-2.7,-1.3) node {{ $E_{9}$}};
   \draw (-1.7,-1.3) node {{ $E_{10}$}};
   \draw (3.7,-1.3) node {{ $E_{11}$}};
   \draw (2.7,-1.3) node {{ $E_{12}$}};
   \draw (.7,-1.3) node {{ $E_{13}$}};
   \draw (1.7,-1.3) node {{ $E_{14}$}};
   \filldraw  (1,1) circle (2pt)
              (2,1) circle (2pt)
              (4,1) circle (2pt)
              (1,-1) circle (2pt)
              (3,-1) circle (2pt)
              (4,-1) circle (2pt)
              (-3,1) circle (2pt)
              (-3,-1) circle (2pt);
   \filldraw  [fill=white]  (-2,1) circle (3pt)
                            (-2,-1) circle (3pt);
   \filldraw  [fill=magenta]  (3,1) circle (3pt);
   \filldraw  [fill=grey]  (0,0) circle (3pt);
   \filldraw  [fill=green]  (2,-1) circle (3pt);
   \filldraw  [fill=red]  (-1,0) circle (3pt);
   \end{tikzpicture}
\end{tabular}

\end{center}

\noindent For simplicity we denote the divisors associated with this resolution as

\vskip 2mm

$\cdot$ $F_1= (3,6,8,15,24,3,3,6,3,6,3,6,9,15)$,

$\cdot$ $F_2=(4,4,4,8,12,8,9,18,9,18,4,8,12,20)$,

$\cdot$ $F_3=(5,5,5,10,15,5,5,10,5,10,7,14,20,35)$.

\vskip 2mm

\noindent In the same manner, the relative canonical divisor is $K_{\pi}=(1,2,3,6,10,2,3,6,3,6,2,4,6,10)$.
The Newton nest of $\fab$ consists of the exceptional divisors $E_1, E_{5}, E_{6}$ and $E_{14}$ and they correspond, matching the colors, to the $\cC$-facets
of the log-canonical wall which is:

\begin{center}

\begin{tikzpicture}
\begin{axis}[view/h=120]
\addplot3 [patch, magenta]
table {
x y z

0.45833333333333 0 0
0.41025641025641 0.346153846153846 0
0.33333333333333 0 0.2
};
\addplot3 [patch, grey]
table {
x y z

0.41025641025641 0.346153846153846 0
0.33333333333333 0 0.2
0 .25 .2
};
\addplot3 [patch, red]
table {
x y z

0 0.375 0
0.41025641025641 0.346153846153846 0
0 .25 .2
};
\addplot3 [patch, green]
table {
x y z

0 0 0.342857142857143
0.33333333333333 0 0.2
0 .25 .2
};
\end{axis}
\end{tikzpicture}
\end{center}


%

\end{example}

In the following example we present a case where $X$ has a log-canonical singularity
and Theorem \ref{thm:Pi} does not hold. This shows how sharp is the condition of having
points in the log-canonical wall with multiplicity one.

\begin{example}\label{ex:Sing_racional_2}

Consider a surface $X$  with a rational singularity at $O$ whose
minimal resolution $\pi:X' \lra X$ has six exceptional components $E_1,\dots, E_6$
with the following dual graph and intersection matrix:

\begin{center}

   \begin{tikzpicture}[scale=0.9]
   \draw  (0,0) -- (3,0);
   \draw  (1,0) -- (0,1);
   \draw  (2,0) -- (3,1);
   \draw [dashed,->,thick] (0,0) -- (-1,1);
   \draw [dashed,->,thick] (0,0) -- (-1.2,0.67);
   \draw [dashed,->,thick] (0,0) -- (-1.3,0.33);
   \draw [dashed,->,thick] (0,0) -- (-1.4,0);
   \draw [dashed,->,thick] (3,0) -- (4,1);
   \draw [dashed,->,thick] (3,0) -- (4.4,0);
   \draw [dashed,->,thick] (0,1) -- (-1.2,1.67);
   \draw [dashed,->,thick] (1,0) -- (0.5,1.1);
   \draw [dotted,thick] (3.75,0) arc [radius=.75, start angle=0, end angle= 45];
   \draw (1.8,-0.3) node {{\small $E_1$}};
   \draw (0.8,-0.3) node {{\small $E_2$}};
   \draw (2.8,1.3) node {{\small $E_3$}};
   \draw (2.8,-0.3) node {{\small $E_4$}};
   \draw (0,1.3) node {{\small $E_5$}};
   \draw (-.2,-0.3) node {{\small $E_6$}};
   \draw (3.9,0.3) node {{\tiny $12$}};
   \filldraw  (3,1) circle (2pt);
   \filldraw  [fill=white]  (0,0) circle (3pt)
                            (1,0) circle (3pt)
                            (2,0) circle (3pt)
                            (0,1) circle (3pt)
                (3,0) circle (3pt);
   \end{tikzpicture}\\ \parbox{30mm}{\begin{center} Vertex ordering\end{center}}\\

   \end{center}

  $$M=(E_i\cdot E_j)_{1\leqslant i,j\leqslant 6}={
  \left( \begin {array}{cccccc} -2&1&1&1&0&0\\ \noalign{\medskip}1&-3&0&0&1&1\\ \noalign{\medskip}1&0&-1&0&0&0\\
\noalign{\medskip}1&0&0&-3&0 &0\\ \noalign{\medskip}0&1&0&0&-3&0\\
\noalign{\medskip}0&1&0&0&0&-6
\end {array} \right)}.$$

\vskip 2mm

The fundamental cycle is the divisor $Z = (3, 2, 3, 1, 1, 1)$ and the relative canonical divisor is
$K_\pi=\left(-\frac{1}{2},-1,\frac{1}{2},-\frac{1}{2},-\frac{2}{3},-\frac{5}{6}\right)$ so the singularity is log-canonical. Then we consider a
duple of ideals $\fab=(\A_1,\A_2)$, with $\A_1$ non singular 
and $\A_2=\m$ given by the  divisors  $F_1=(15, 6, 15, 9, 2, 1)$ and
$F_2=Z = (3, 2, 3, 1, 1, 1)$. The log-canonical wall has two $\cC$-facets and the corresponding mixed multiplier ideals are different (see
Figure 2).
In particular we have jumping points on the log-canonical wall with multiplicity bigger that $1$. In this case the Newton nest consists of the exceptional divisors
$E_1, E_2$ and $E_4$ so we no longer have
the bijection  given in Theorem \ref{thm:Pi}.

\begin{figure}[ht!!!]
  \begin{center}
  \medskip
    \includegraphics[width=.40\textwidth]{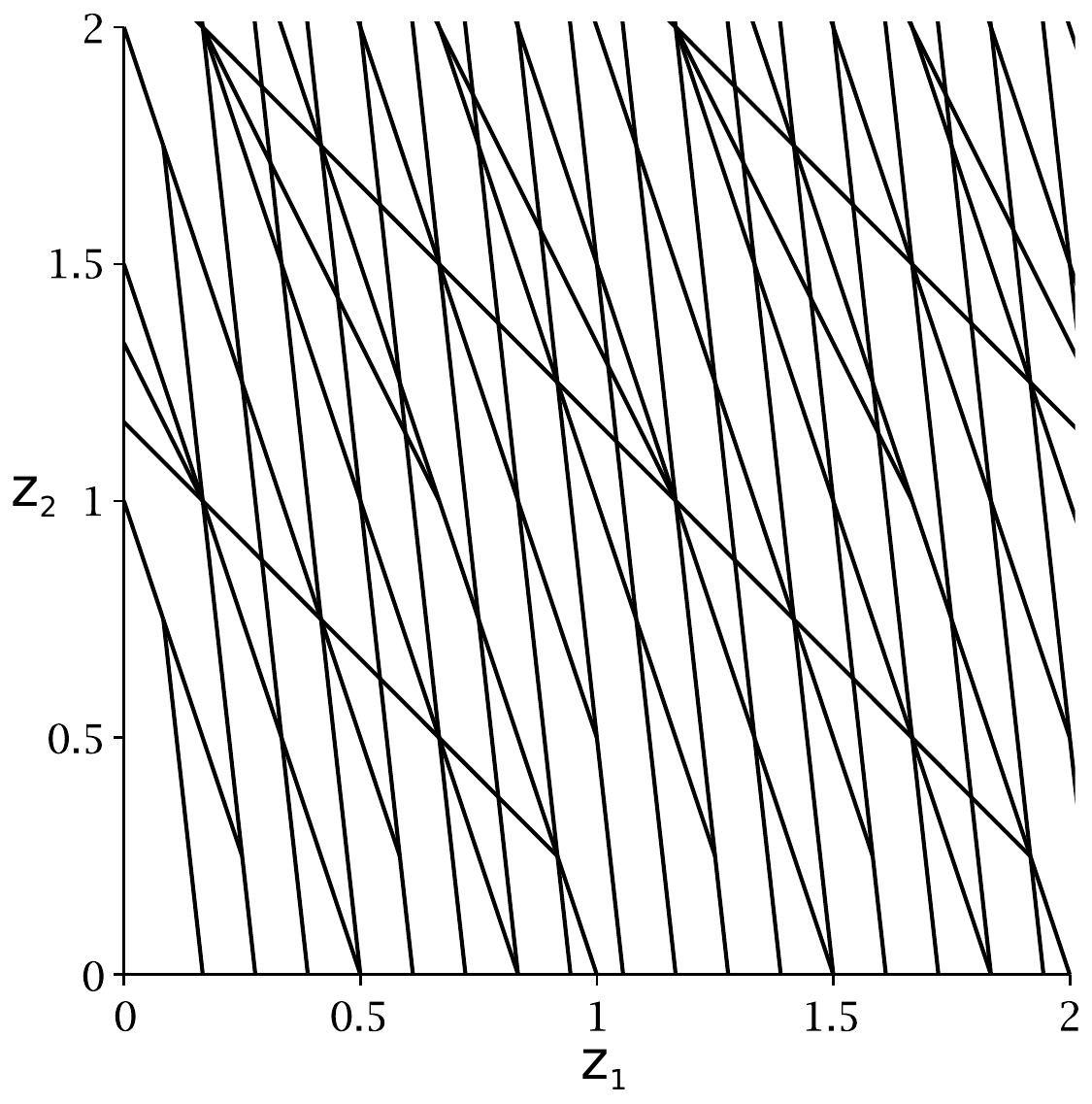}
  \end{center}
\caption{Constancy regions of the the mixed multiplier ideals of $\fab$. }
\end{figure}
\end{example}

\end{document}